\documentclass[11pt, oneside]{amsart}

\newcommand{\hide}[1]{}

\usepackage{amssymb}
\usepackage{graphicx}
\usepackage{amsmath}

\usepackage{geometry}
\usepackage{caption}

\usepackage{float}

\usepackage{bm}

\DeclareMathAlphabet{\mathdutchcal}{U}{dutchcal}{m}{n}
\SetMathAlphabet{\mathdutchcal}{bold}{U}{dutchcal}{b}{n}
\DeclareMathAlphabet{\mathdutchbcal}{U}{dutchcal}{b}{n}

\usepackage{tikz}
\usetikzlibrary{arrows, arrows.meta}

\def\textcolor#1{}

\newcommand{\R}{\mathbb{R}}

\newcommand{\dH}{d_{\mathcal H}}

\newcommand{\Ll}{\mathcal{I}}
\newcommand{\Rr}{\mathdutchcal{in}}
\newcommand{\Ch}{\mathdutchcal{Ch}}
\newcommand{\Hc}{\mathcal H}
\newcommand{\F}{\mathcal{F}}
\newcommand{\B}{{\bm B}}

\renewcommand{\tilde}{\widetilde}
\renewcommand{\rho}{\varrho}
\renewcommand{\phi}{\varphi}
\newcommand{\eps}{\varepsilon}
\renewcommand{\theta}{\vartheta}

\DeclareMathOperator{\inter}{int}

\newcommand{\sm}{\setminus}

\renewcommand{\ge}{\geqslant}
\renewcommand{\le}{\leqslant}

\usepackage[colorlinks=true,urlcolor=blue]{hyperref}

\theoremstyle{theorem}
\newtheorem{theorem}{Theorem}[section]
\newtheorem{lemma}[theorem]{Lemma}
\newtheorem{proposition}[theorem]{Proposition}
\newtheorem{corollary}[theorem]{Corollary}

\newtheorem{MainTheorem}{Theorem}

\newtheorem*{theoremRIP}{The Reverse Isoperimetric Inequality}

\newtheoremstyle{Intro}% hnamei
{}% hSpace abovei
{}% hSpace below
{\itshape}% hBody fonti
{}% hIndent amounti
{\scshape}% hTheorem head fonti
{.}% hPunctuation after theorem headi
{.5em}% hSpace after theorem headi
{}% hTheorem head spec (can be left empty, meaning ‘normal’)i

\theoremstyle{Intro}

\theoremstyle{definition}

\theoremstyle{remark}
\newtheorem*{remark}{\textsc{Remark}}

\newcounter{reminder}

\newtheoremstyle{claim}% name of the style to be used
  {}% measure of space to leave above the theorem. E.g.: 3pt
  {}% measure of space to leave below the theorem. E.g.: 3pt
  {\itshape}% name of font to use in the body of the theorem
  {0pt}% measure of space to indent
  {\scshape}% name of head font
  {.}% punctuation between head and body
  { }% space after theorem head; " " = normal interword space
  {\thmname{#1}\thmnumber{ #2}\thmnote{ (#3)}}

\theoremstyle{claim}
\newtheorem{claim}[theorem]{Claim}

\numberwithin{equation}{section}

\title[Stability of planar reverse isoperimetric inequalities]{Stability of reverse isoperimetric inequalities in the plane: area, Cheeger, and inradius}

\author[Kostiantyn Drach]{Kostiantyn Drach}
\author[Kateryna Tatarko]{Kateryna Tatarko}
\thanks{\ The first author is partially supported by Departament de Recerca i Universitats de la Generalitat de Catalunya (2021 SGR 00697), Agencia Estatal de Investigaci\'on Grant PID2023-147252NB-I00, and Maria de Maeztu Excellence Grant CEX2020-001084-M. The second author is partially supported by NSERC Discovery Grant number 2022-02961. The authors are grateful to Julian Scheuer who asked the question about stability in the reverse isoperimetric inequality during the visit of the first author to the University of Frankfurt. 
We are grateful to Beatrice-Helen Vritsiou, Paul Milan Diaz, and Teresa Norris for identifying an issue in the proof in an earlier version of Section 4.2 and suggesting a correction. We also thank the anonymous referees for careful reading and helpful comments.}

\date{}

\address{Universitat de Barcelona, Gran Via de les Corts Catalanes, 585, 08007 Barcelona, Spain}
\address{Centre de Recerca Matem\`atica, Edifici C, Carrer de l'Albareda, 08193 Bellaterra, Barcelona, Spain}
\email{kostiantyn.drach@ub.edu}
\address{Department of Pure Mathematics, University of Waterloo, Waterloo, ON, N2L 3G1, Canada}
\email{ktatarko@uwaterloo.ca}

\begin{document}

\begin{abstract}
In this paper, we present stability results for various reverse isoperimetric problems in $\R^2$. Specifically, we prove the stability of the reverse isoperimetric inequality for \emph{$\lambda$-convex bodies} --- convex bodies 
with the property that each of their boundary points $p$ supports a ball of radius $1/\lambda$ so that the body lies inside the ball in a neighborhood of $p$. For convex bodies with smooth boundaries, $\lambda$-convexity is equivalent to having the curvature of the boundary bounded below by $\lambda > 0$. Additionally, within this class of convex bodies, we establish stability for the reverse inradius inequality and the reverse Cheeger inequality. Even without its stability version, the sharp reverse Cheeger inequality is new in dimension $2$.  
\end{abstract}

\maketitle

\section{Introduction}

Let  $\R^n$ be the  $n$-dimensional Euclidean space. A \emph{convex body} is a compact, convex subset of $\R^n$ with non-empty interior.   The classical isoperimetric inequality asserts that among all domains $\Omega \subset \R^n$ of a given surface area $|\partial \Omega|$, the ball has the largest possible volume and it is a unique domain with this property. In this context, it is natural to ask the following \emph{stability question}: 

\smallskip
{\it Is it true that $\Omega$ is close to a ball (say, in the Hausdorff distance) provided that the volume of $\Omega$ is sufficiently close to the volume of the ball of surface area $|\partial \Omega|$?}
\smallskip

Over the years, various quantitative stability versions of the isoperimetric inequality have been established. For example, if the closeness of bodies is measured in the Hausdorff distance, the stability in the isoperimetric inequality is proven by Diskant \cite{Di} and Groemer \cite{Gr2}. Fusco, Maggi, and Pratelli obtained the optimal stability result in terms of the volume difference in \cite{FMP}. The generalizations of the latter stability result to the anisotropic isoperimetric inequality and Brunn--Minkowski inequality were obtained by Figalli, Maggi, and Pratelli in \cite{FiMP, FiMP2}. Recently, the study of stability of geometric and functional inequalities garnered a lot of attention. Such quantitative results are not only important in their own right but also lead to the development of new methods at the interface of geometry and analysis. We refer the reader to the surveys of Groemer \cite{Gr} and Figalli \cite{Fi}, and references therein.

In contrast to the classical problem, the \emph{reverse isoperimetric problem} is concerned with finding a domain with the \emph{smallest} volume among all domains of a given surface area. As stated, the problem has a trivial solution. Thus, a crucial initial step is to identify a natural setting that makes the reverse isoperimetric problem both meaningful and non-trivial. In \cite{Bal, Bal2}, K.~Ball established the reverse isoperimetric inequality in $\mathbb{R}^n$ by restricting the problem to affine classes of convex bodies with the equality case settled later by Barthe \cite{Bar}. A stability version of Ball's inequality for symmetric convex bodies was obtained by B\"or\"oczky, Fodor and Hug in \cite{BFH}, and for general convex bodies in \cite{BH} by B\"or\"oczky and Hug.

Another fruitful setting for tackling the reverse isoperimetric problem is to impose some (local) curvature restrictions on $\partial \Omega$. This approach has received a lot of attention recently. In \cite{HTr}, Howard and Treibergs obtained a reverse isoperimetric inequality with equality characterization for planar curves (not necessarily convex) with curvature $\kappa$ satisfying $|\kappa| \le 1$. 
Chernov, Drach, and Tatarko \cite{CDT} fully resolved the reverse isoperimetric problem in $\R^n$, $n \ge 2$, for convex bodies with principal curvatures of their boundaries uniformly bounded \emph{from above} (the so-called \emph{$\lambda$-concave bodies}). An alternative proof of the latter result and its generalization to the Minkowski relative geometry was obtained by  Saor\'in G\'omez and Yepes Nicol\'as in \cite{SY} (see also \cite{Piotr} and \cite{Jog} for related results).

For the other direction of the curvature bounds, the reverse isoperimetric problem for convex curves with curvature bounded \emph{below} (the so-called \emph{$\lambda$-convex bodies}, see the definition below) was resolved by Borisenko and Drach in \cite{BorDr14}. See also \cite{FKV} for an alternative proof by Fodor, Kurusa, and Vigh without an equality characterization. Recently, in \cite{DrTa}, Drach and Tatarko provided another proof of the latter result and extended it to $\R^3$.

In this paper, we make the first step in the study of the stability of reverse-type isoperimetric inequalities with curvature constraints. Namely, we obtain the stability version of the reverse isoperimetric inequality in $\R^2$ for $\lambda$-convex bodies by exploiting the approach from~\cite{DrTa}. Furthermore, we establish two closely related stability results in the class of planar $\lambda$-convex bodies:

\begin{itemize}
    \item A stability version of the reverse inradius inequality obtained in \cite{DrTa} (see \cite{MilInradius} for the original proof).   
    \item A stability version of the reverse Cheeger inequality;
    the latter inequality, which has not been previously established for the class of $\lambda$-convex bodies, is proven here as an initial step towards addressing stability.   
\end{itemize}

\subsection{Main results}
For a given $\lambda > 0$, a convex body $K\subset \R^n$ is called \emph{$\lambda$-convex} if for every $p \in  \partial K$ there exists a  neighborhood $U_p$ and a ball $B_{\lambda, p}$ of radius $1\slash \lambda$ passing through $p$ in such a way that
$$
U_p \cap K \subset B_{\lambda, p}
$$
(see Figure~\ref{Fig:Def}).
A \emph{$\lambda$-convex lens} in $\R^n$ is the intersection of two balls of radius $1\slash \lambda$. 

\begin{figure}[h!] 
    \centering
     \includegraphics[scale=0.25]{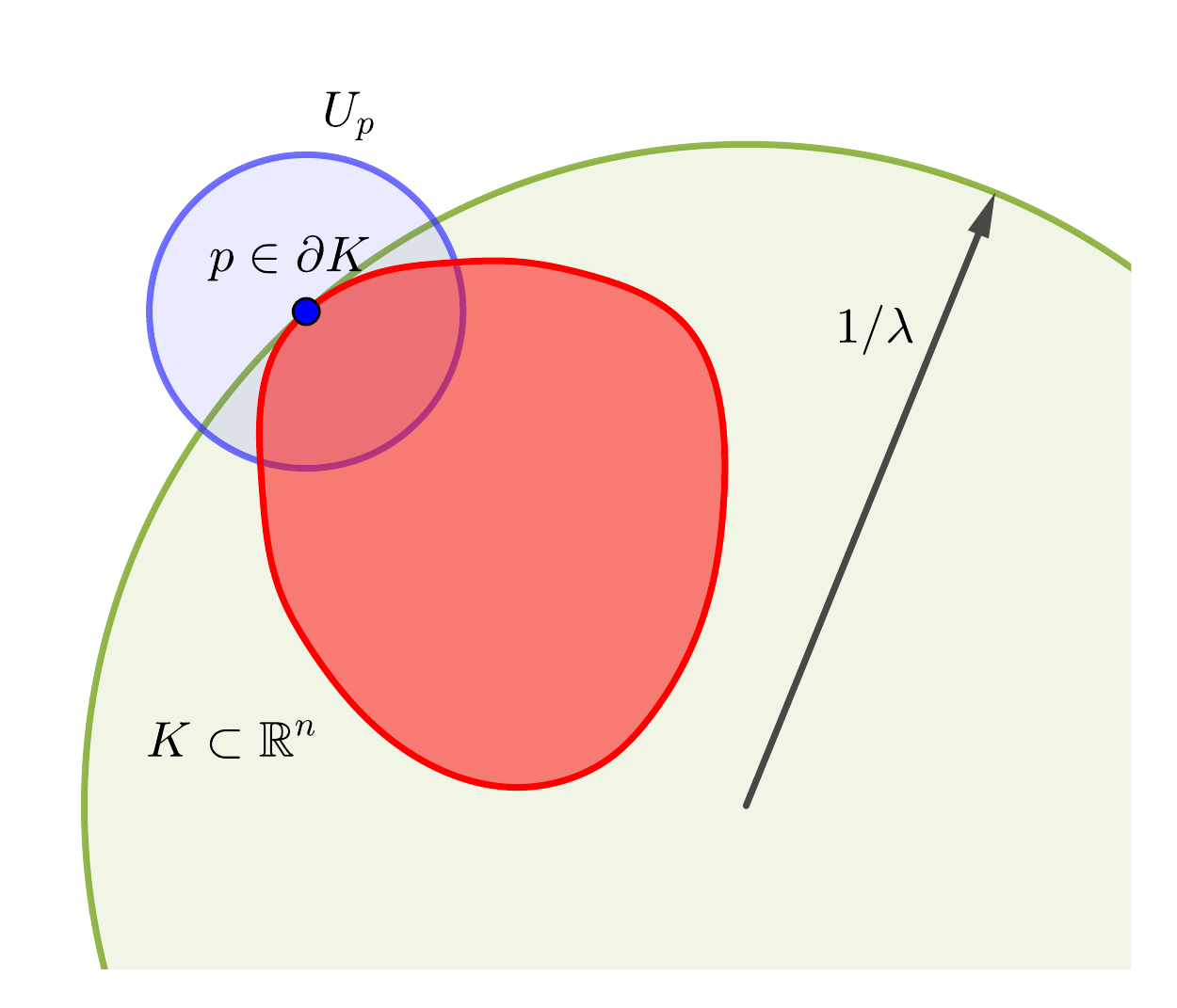}
   \caption{The definition of a $\lambda$-convex body}
    \label{Fig:Def}
\end{figure}

Recently, the class of  $\lambda$-convex bodies has been extensively studied in convex and discrete geometry, for example,  see \cite{BlaLeb, Bezdek2012, FV, JMR, NV, PP}. One of the typical examples of $\lambda$-convex bodies is the finite intersection of Euclidean balls of the fixed radius $1\slash \lambda$, which naturally arise in connection with the Kneser--Poulsen conjecture \cite{Bezdek2008}.  
We also note that the notion of $\lambda$-convexity appeared in the literature under various names such as ``$r$-hyperconvexity", ``spindle convexity", ``ball convexity", and ``$r$-duality". For more details and references, we refer the reader to the nice surveys \cite{BLN, BLNP} about $\lambda$-convexity and related topics.

By Blaschke's rolling theorem (see, e.g., \cite{DrBla} and references therein), if $K \subset \R^n$ is $\lambda$-convex, then $B_{\lambda, p} \supset K$ for every $p \in \partial K$. In particular, $|K| \le |B_\lambda|$ and $|\partial K| \le |\partial B_\lambda|$, where $B_\lambda \subset \R^n$ is a ball of radius $1/\lambda$. Thus we can define a function 
\[
\F \colon \big[0, |\partial B_\lambda|\big] \to \big[0, |B_\lambda|\big], \quad x \mapsto \F(x),
\]
where $\F(x)$ is the volume of a $\lambda$-convex lens of surface area $x$. Using this function, for an arbitrary $\lambda$-convex body $K \subset \R^n$, we define the \emph{lens isoperimetric quotient} as
\[
\Ll(K) := \frac{|K|}{\F(|\partial K|)}.
\] 
Note that $\Ll(L) = 1$ for every $\lambda$-convex lens, and this value is the minimum of $\Ll(K)$ among all $\lambda$-convex bodies as was shown in $\R^2$ \cite{BorDr14, DrTa, FKV}  and in $\R^3$ \cite{DrTa}. We note that the problem of minimizing $\Ll(K)$ among all  $\lambda$-convex bodies remains open in $n \ge 4$. 

\begin{theoremRIP}
For every $\lambda$-convex body $K \subset \R^n, n \in \{2,3\}$, 
\[
\Ll(K) \ge 1,
\]
and equality is possible if and only if $K$ is a $\lambda$-convex lens.
\end{theoremRIP}

One of the main results of this paper is the corresponding stability result on the plane, which says that if a $\lambda$-convex body has its lens quotient sufficiently close to $1$, then the body is close to a lens in  Hausdorff distance.  

\begin{MainTheorem}[Stability in the reverse isoperimetric inequality in $\R^2$]
\label{Thm:A}
For every $\lambda > 0$ there exist $\eps_0  \in  (0,1)$ and $C \ge 1$ such that for every positive $\eps \le \eps_0$ the following holds: 

If $K \subset \R^2$ is a $\lambda$-convex body such that
\begin{equation}
\label{Eq:Approx}
\Ll(K) \le 1+\eps,
\end{equation}
then there exists a $\lambda$-convex lens $L \subset \R^2$ such that
\[
\dH(K, L) \le C \cdot \eps^{1/4},
\]
where $\dH$ is the Hausdorff distance.
\end{MainTheorem}

Recall that the Hausdorff distance of convex bodies $K, L$ in $\R^n$ is defined as
$$
\dH(K,L) = \min\{\mu \ge 0: \ K \subset L + \mu \B, \ L \subset K + \mu \B \},
$$ 
where $\B$ is the unit Euclidean ball in $\mathbb{R}^n$.

Since for a given $\lambda>0$, the set of $\lambda$-convex domains that contains the origin is compact in the space of all convex sets with respect to the Hausdorff distance, we get the following corollary of Theorem~\ref{Thm:A}:

\begin{corollary}
There exist constants $C>0$ and $\eps_0 > 0$ that depend only on $\lambda$ so that for every $\lambda$-convex body $K \subset \mathbb R^2$ with $\Ll(K) \le 1+\eps_0$, there exists a $\lambda$-convex lens $L \subset \mathbb R^2$ such that
\[
\dH(K, L) \le C \cdot \left(\mathcal I(K) - 1\right)^{{1}/{4}}. 
\]
\end{corollary}

%%%%%%%%Cheeger

Let $\Omega \subset \R^2$ be a closed bounded domain. The \emph{Cheeger constant of $\Omega$} is defined as
\[
h(\Omega) := \inf \limits_{X \subset \Omega} \frac{|\partial X|}{|X|},
\]
where the infimum is taken over all measurable sets $X \subset \Omega$ with rectifiable boundary. The classical inequality due to Cheeger \cite{C} in $\R^2$ can be formulated as: if $B \subset \R^2$ is a Euclidean ball, and 
\[
\text{if }|\partial \Omega| = |\partial B|, \text{ then } h(\Omega) \ge h(B), 
\]
and the equality holds if and only if $\Omega$ is a ball of the same radius as $B$ (see also \cite{Pa}). 

The second main result of the paper is a reverse form of Cheeger's inequality, along with its stability version. In order to state it, we define for a $\lambda$-convex domain $K$ the \emph{lens Cheeger isoperimetric quotient} as
\[
\Ch(K) := \frac{h(K)}{\Hc(|\partial K|)}, 
\]
where $\Hc \colon x \mapsto \Hc(x)$ is the Cheeger constant of the $\lambda$-convex lens of surface area $x$.

\begin{MainTheorem}[Reverse Cheeger inequality and its stability]
\label{Thm:C}
    For every $\lambda$-convex body $K \subset \R^2$, its lens Cheeger isoperimetric quotient satisfies
    \[
    \Ch(K) \le 1,
    \]
    and the equality holds if and only if $K$ is a $\lambda$-convex lens. 
    
    Moreover, there exist $\eps_0 \in (0,1)$ and a constant $C \ge 1$ such that for every positive $\eps \le \eps_0$, if
    \[
    \Ch(K) \ge 1 - \eps,
    \]
    then 
    \[
    \dH(K,L) \le C \cdot \eps^{{1}/{4}}
    \]    
    for some $\lambda$-convex lens $L \subset \R^2$.
\end{MainTheorem}

\begin{remark}
In the proof of Theorem~\ref{Thm:C}, we will use the relation between the Cheeger constant of a convex body $K$ and the geometry of inner parallel sets of $K$, as shown in \cite{KLR}. Using the same connection, a version of a reverse Cheeger inequality for planar bodies of constant width $w$  was recently proved in \cite{Bog}  (see \cite{HL} for the original proof of this inequality). It was shown that the Reuleaux triangle maximizes the Cheeger constant among all bodies of a given constant width. It is interesting to note that each body of constant width $w$ is $1/w$-convex.   
\end{remark}

Finally, the last main result of the paper is the stability version of the reverse inradius inequality. Recall that for a convex body $K$, its \emph{inradius} $r(K)$ is the radius of a largest ball contained in $K$. By Blaschke's rolling theorem, $r(K) \le 1/\lambda$ provided $K$ is $\lambda$-convex. 

Similar to the lens isoperimetric quotient, for a $\lambda$-convex body $K\subset \R^n$ we can define \emph{lens inradius isoperimetric quotient} as
\[
\Rr(K) := \frac{r(K)}{\mathcal G(|\partial K|)},
\] 
where $r(K)$ is the inradius of $K$ and 
\[
\mathcal G \colon \big[0, |\partial B_\lambda|\big] \to \big[0, 1/\lambda\big], \quad x \mapsto \mathcal G(x),
\]
where $\mathcal G(x)$ is the inradius of a $\lambda$-convex lens of surface area $x$. In \cite{Dr, DrTa, MilInradius} it was shown that for every $\lambda$-convex body $K \subset \mathbb R^n$, its lens inradius quotient satisfies the inequality (the reverse inradius inequality)
\[
\Rr(K) \ge 1.
\]
 Moreover, equality is possible if and only if $K$ is a  $\lambda$-convex lens. We establish the corresponding stability result:

\begin{MainTheorem}[Stability in reverse inradius inequality in $\R^2$]
\label{Thm:B}
For every $\lambda > 0$ there exist $\eps_0 \in (0,1)$ and $C \ge 1$ such that for every positive $\eps \le \eps_0$ the following holds: 

If $K \subset \R^2$ is a $\lambda$-convex body with $\lambda > 0$, such that
\begin{equation}
\label{Eq:Approx}
\Rr(K) \le 1 + \eps,
\end{equation}
then there exists a $\lambda$-convex lens $L \subset \R^2$ such that
\[
\dH(K, L) \le C \cdot \sqrt{\eps}.
\]
\end{MainTheorem}

In our stability theorems above, we do not claim that the dependence on $\eps$ is the best possible. It would be interesting to find the optimal order of $\eps$ (see the discussion in Subsection~\ref{SSec:Optimal}).

The proofs of Theorems~\ref{Thm:A} and~\ref{Thm:C} rely on Theorem~\ref{Thm:B}. Hence, we start by establishing Theorems~\ref{Thm:A} (Section~\ref{Sec:ProofThmA}) and~\ref{Thm:C} (Section~\ref{Sec:ProofThmC}) assuming that Theorem~\ref{Thm:B} is true, and then provide the proof of Theorem~\ref{Thm:B} in Section~\ref{Sec:ProofB}.  

Finally, it seems plausible that the methods of this paper can be used to deduce stability for the full range of results established in \cite{DrTa}, i.e., to show stability in the reverse isoperimetric inequality in $\R^3$ and the reverse inradius inequality in $\R^n$, $n \ge 3$. 

%%%%%%%%%%%%%%%%%%%%%%%%%%%%%%%%%%%%%%%%%%%%%%%%%
\section{Proof of Theorem~\ref{Thm:A} assuming Theorem~\ref{Thm:B}}
\label{Sec:ProofThmA}
%%%%%%%%%%%%%%%%%%%%%%%%%%%%%%%%%%%%%%%%%%%%%%%%%
Before proving Theorem~\ref{Thm:A} we need the following preliminary section about inner parallel bodies.

\subsection{Inner parallel bodies in $\mathbb{R}^n$.}
Let $A$ be a convex body in $\R^n$. The inner parallel body $A_t$ at distance $t \geq 0$ is defined as
$$
A_t = A \sim t \B := \{x \in \R^n : x + t \B \subset A\}, 
$$
where $\B \subset \mathbb{R}^n$ is the unit Euclidean ball. The largest $t$ for which $A_t$ is not empty is the inradius $r(A)$ of $A$.

\begin{proposition}
\label{Prop:Concave}
    Let $A$ be a convex body in $\R^2$. The perimeter $|\partial A_t|$ is a decreasing and concave function in $t \in [0, r(A)]$.
\end{proposition}
\begin{proof}
The system of inner parallel bodies $t \mapsto A_t$ is concave with respect to the inclusion (see \cite[Lemma 3.1.13]{Sch}). In particular,
$$
\frac12 A_{t} + \frac12 A_{s} \subset A_{\frac{t+s}2} 
$$
for any $t, s \in [0, r(A)]$. Thus,
$$
\frac12  |\partial A_t| + \frac12 |\partial A_s|  \leq \left|\partial \left(\frac12 A_{t} + \frac12 A_{s}\right)\right| \leq \left|\partial A_{\frac{t+s}2}\right| 
$$
where the first inequality follows from the generalized Brunn--Minkowski inequality for the intrinsic volumes (see \cite[Theorem 7.4.5]{Sch}). 
\end{proof}

We note that the concavity property does not necessarily hold in $\R^n$, $n \geq 3$. In particular, a simple computation reveals that $|\partial L_t|$ is convex for $t \in [0, r(L)]$ and for any $\lambda$-convex lens $L \subset \R^3$.

\subsection{Proof of Theorem \ref{Thm:A}}
Let $K$ be a $\lambda$-convex body that satisfies the assumption of Theorem~\ref{Thm:A}. Choose a $\lambda$-convex lens $L$ such that $|\partial L| = |\partial K|$ which implies that $\F(|\partial K|) = |L|$. Hence, by assumption, $\Ll(K) = \frac{|K|}{|L|} \le 1 + \varepsilon$, or equivalently, 
\begin{equation}
\label{Eq:AssThmA}
|K| - |L| \le \varepsilon |L|.    
\end{equation}

We claim that if $\Ll(K) \le 1 + \varepsilon$ then
$\Rr(K) \le 1 + C\sqrt\varepsilon$ for some $C>0$.

\begin{figure}
	\includegraphics[scale=1.1]{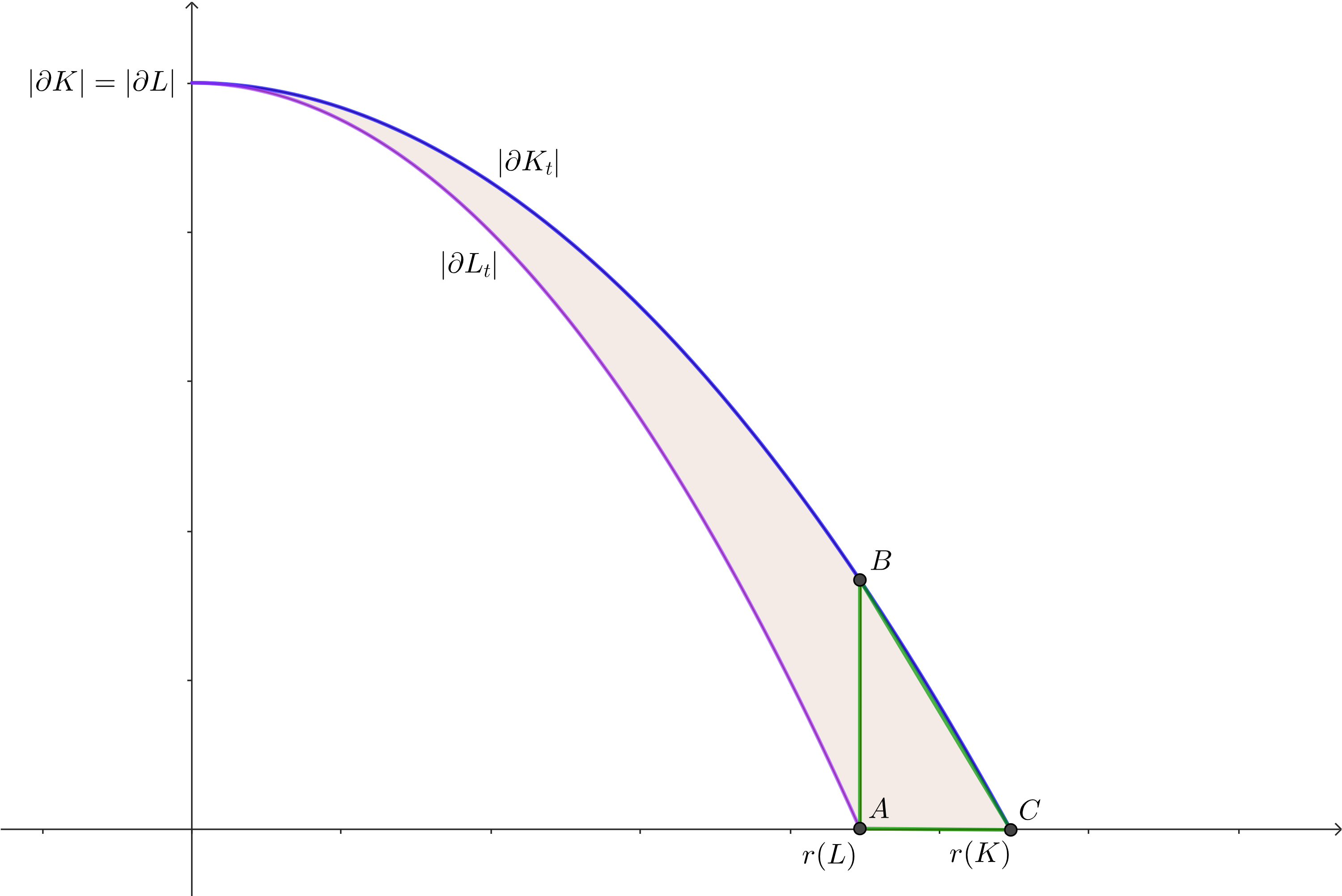}
	\caption{Graphs of $|\partial K_t|$ and $|\partial L_t|$ in the proof of Theorem~\ref{Thm:A}}
	\label{Fig:Stability}
\end{figure}

It was shown in \cite[Appendix 1]{DrTa} that 
\begin{equation}
\label{Eq:DT}
    |\partial K_t| > |\partial L_t| \quad \text{for all} \quad t \in (0, r(L)], \quad \text{provided} \quad |\partial K| = |\partial L|
\end{equation}
(see Figure~\ref{Fig:Stability}). Since the volume of inner parallel body is differentiable (\cite{Ma}), we also have that 
\begin{equation}\label{eq:volume}
	|K| = \int_0^{r(K)} |\partial K_t| \, dt \quad\text{ and }\quad|L| = \int_0^{r(L)} |\partial L_t| \, dt.
\end{equation}
In order to estimate $\Rr(K)$, we consider a triangle $\Delta \, ABC$ where
$A = (r(L),0)$, $B = (r(L), |\partial K_{r(L)}|)$ and $C = (r(K), 0)$ (see Figure~\ref{Fig:Stability}).  Then the area of this triangle is 
$$
|\Delta ABC| = \frac12 (r(K) - r(L)) \cdot |\partial K_{r(L)}|.
$$ 

By Proposition~\ref{Prop:Concave} and \eqref{Eq:DT}, $\Delta ABC$ is completely contained between the graphs of $|\partial K_t|$ and $|\partial L_t|$. By \eqref{Eq:DT} and \eqref{eq:volume}, the area between the graphs $|\partial K_t|$ and $|\partial L_t|$ is $|K|-|L|$.  Therefore, 
\begin{equation}\label{eq:area_upp}
	|\Delta ABC| \le |K| - |L| \le \varepsilon |L| 
\end{equation}
where we used inequality~\eqref{Eq:AssThmA}. Now we estimate the area of $\Delta ABC$ from below 
$$
|\Delta ABC|  \ge  \frac12 \frac{(r(K) - r(L))^2}{r(K)} |
\partial L|
$$
where we used the fact that $\frac{r(K) - r(L)}{r(K)} K \subset K_{r(L)}$ (see \cite[Eq. (2.4)]{HS}), monotonicity of the perimeter, and  $|\partial K| = |\partial L|$. Hence,
$$
 \frac12 \frac{(r(K) - r(L))^2}{r(K)} \le \varepsilon \frac{|L|}{|
\partial L|}.
$$
Simplifying and solving a quadratic inequlaity for $\frac{r(K)}{r(L)}$, we get
\begin{equation}
	\Rr(K) = \frac{r(K)}{r(L)} \le \frac12 \left(2 + \frac{2 |L| \varepsilon}{|\partial L| r(L)} + \sqrt{\left(2 + \frac{2 |L|\varepsilon }{|\partial L| r(L)} \right)^2 - 4} \right) \le 1 + C \sqrt{\varepsilon} 
\end{equation}
where $C = \frac{2 |L|}{|\partial L| r(L)} > 0$. We also observe that $C$ is bounded from above by expressing the area $|L|$ and the inradius $r(L)$ of the lens $L$ in terms of its perimeter $|\partial L|$ (a straightforward computation):
\begin{equation}
\label{Eq:Lens}
C = \frac{2 |L|}{|\partial L| r(L)} = \frac{|\partial L| - 2 \sin \frac{|\partial L|}2}{|\partial L| \left(1 - \cos \frac{|\partial L|}4\right)} < 2.
\end{equation}
By Theorem~\ref{Thm:B}, there exists a constant $\widetilde{C} >1$ and a convex lens $\widetilde{L}$ in $\mathbb{R}^2$ such that $ \dH(K, \widetilde{L}) \le \widetilde{C} \cdot \sqrt[4]{\eps}.$

\subsection{On the optimality of the exponent of $\eps$}
\label{SSec:Optimal}

The optimal exponent in the stability estimate 
\[
\dH(K,L) \le C \cdot \eps^\alpha \quad \text{assuming} \quad \Ll(K) \le 1+\eps
\] should satisfy 
$\alpha \in [1/4,1/2].$
Indeed, the lower bound is given by Theorem~\ref{Thm:A}, and the upper bound is given by the following example. Assume $\lambda=1$, and let $K_0$ be a disk of radius $1-\delta$ for some small $\delta >0$. Then $K_0$ is a $1$-convex body with 
\[
|K_0| = \pi (1-\delta)^2, \qquad |\partial K_0| = 2\pi(1-\delta)
\]
and $\mathcal F(|\partial K_0|) = \pi(1-\delta) - \sin(\pi(1-\delta))$ (compare with \eqref{Eq:Lens}). Hence
\[
\Ll(K_0) = \frac{\pi(1-\delta)^2}{\pi(1-\delta)-\sin(\pi(1-\delta))} = 1+\delta^2+O(\delta^3).
\]
On the other hand, by the symmetry, the Hausdorff distance from $K_0$ to lenses of perimeter $|\partial K_0|$ is minimal for the lens $L_0$ whose center coincides with the center of $K_0$. In this case, $\dH(K_0,L_0) = R(L_0)-(1-\delta)$, where $R(K_0)$ is the half of the diameter of $L_0$. By an explicit computation (using $r(L_0)$, see \eqref{Eq:Lens}), we obtain
\begin{equation*}
\begin{aligned}
    \dH(K_0,L_0) %&= R(L_0)-(1-\delta) 
    = \sqrt{1-(1-r(L_0))^2}-(1-\delta) = \sin \frac{\pi(1-\delta)}{2}-1+\delta = \delta + O(\delta^2).
\end{aligned}    
\end{equation*}
Hence, if we assume that $\Ll(K_0) \le 1+\eps$, then $\dH(K_0,L_0) \le C\sqrt{\eps}$. Thus, $\alpha \le 1/2$. We conjecture that this should be the optimal value of the exponent.

%%%%%%%%%%%%%%%%%%%%%%%%%%%%%%%%%%%%%%%%%%%%%%%%%
\section{Proof of Theorem~\ref{Thm:C} assuming Theorem~\ref{Thm:B}}
\label{Sec:ProofThmC}
%%%%%%%%%%%%%%%%%%%%%%%%%%%%%%%%%%%%%%%%%%%%%%%%%

We will need the following result, which gives an explicit construction for the Cheeger constant of a convex set. Recall that $\Omega_t$ denotes the inner parallel set at distance $t$ of a convex set $\Omega$.

\begin{theorem}[\cite{KLR}]
    \label{Thm:Cheeger}
    Let $\Omega \subset \R^2$ be a convex set. There exists a unique value $t_\star > 0$ such that $|\Omega_{t_\star}| = \pi \cdot t_\star^2$. Then $h(\Omega) = 1/t_\star$.
    %\qed 
\end{theorem}

%%%%%%%%%%%%%%%%%%%%%%%%%%%%%%%%%%%%%%%%%
\subsection{Proof of the reverse Cheeger inequality}
%%%%%%%%%%%%%%%%%%%%%%%%%%%%%%%%%%%%%%%%%

Let $K$ be a $\lambda$-convex set. Without loss of generality, we set $\lambda = 1$. We start by showing that $\Ch(K) \le 1$ with equality if and only if $K$ is a $1$-convex lens. This is equivalent to showing that 
\begin{equation}
\label{Eq:CAss}
h(K) \le h(L), \text{ provided }L\text{ is a $1$-convex lens with }|\partial K|=|\partial L|.
\end{equation}

For $K$ and $L$ satisfying assumption \eqref{Eq:CAss}, as we saw in Section~\ref{Sec:ProofThmA}, the functions
\begin{equation*}
    \begin{aligned}
        &f_K(t) \colon [0, r(K)] \to \left[0, |\partial K|\right], \quad f_K \colon t \mapsto |\partial K_t|,\\   
        &f_L(t) \colon [0, r(L)] \to \left[0, |\partial L|\right], \quad f_L \colon t \mapsto |\partial L_t|\\  
    \end{aligned}
\end{equation*}
are strictly  decreasing and satisfy $f_K(t) \ge f_L(t)$ for all $t \in [0, r(L)]$. Note that $r(K) \ge r(L)$ and this is a strict inequality unless $K$ is a $1$-convex lens (see the discussion before the statement of Theorem~\ref{Thm:B}). Therefore, the functions
\[
g_K(t) := |K_t| = \int_t^{r(K)} f_K(s) \,\text{d}s, \quad g_L(t) := |L_t| = \int_t^{r(L)} f_L(s) \,\text{d}s
\]
are strictly decreasing and satisfy
\[
g_K(t) \ge g_L(t), \quad \forall t \in [0,r(L)],
\]
with equality possible if and only if $K$ is a $1$-convex lens. 

By Theorem~\ref{Thm:Cheeger}, if $t_L = 1/h(L)$, then $g_L(t_L) = \pi t_L^2$, and this is the unique solution of the equation $g_L(t) = \pi t^2$. Thus, $g_L(t) > \pi t^2$ for all $t \in [0, t_L)$. Since $g_K(t) \ge g_L(t)$, we have that $g_K(t) > \pi t^2$ for all $t \in [0, t_L)$ as well. Hence, the solution $t_K$ of the equation $g_K(t) = \pi t^2$ must satisfy $t_K \ge t_L$. Therefore, $h(K) = 1/t_K \le 1/t_L = h(L)$, as desired. 

Finally, if $t_K = t_L$, then $|K_{t_K}| = |L_{t_L}|$, which implies 
\[
\int_{t_K}^{r(K)} f_K(s) \,\text{d}s = \int_{t_L}^{r(L)} f_L(s) \,\text{d}s.
\]
Since $f_K(s) \ge f_L(s)\ge 0$, the equality of those integrals can only occur when $r(K) \le r(L)$. As was discussed above, $r(K) \ge r(L)$, and hence this is possible only if $K$ is a $1$-convex lens. This proves the reverse Cheeger inequality in Theorem~\ref{Thm:B}.

%%%%%%%%%%%%%%%%%%%%%%%%%%%%%%%%%%%%%%%%%
\subsection{Proof of stability in the reverse Cheeger inequality}
%%%%%%%%%%%%%%%%%%%%%%%%%%%%%%%%%%%%%%%%%

Assume that $\Ch(K) \ge 1 - \eps$. Let $L$ be a $1$-convex lens such that $|\partial K| = |\partial L|$. Then, by definition of the Cheeger quotient and by Theorem~\ref{Thm:Cheeger}, 
\[
\frac{h(K)}{h(L)} = \frac{t_L}{t_K} \ge 1-\eps.
\]
Hence, $t_K - t_L \le \eps \cdot t_K$.

By construction, $g_K$ and $g_L$ are strictly decreasing, convex functions. Indeed, $g_K'' = -f_K' > 0$, and similarly for $g_L$. Moreover, since $f_K(t) \ge f_L(t)$, we get 
\[
g'_K(t) = -f_K(t) \le -f_L(t) = g'_L(t).
\]
Hence, 
$
\int_{t_K}^{r(L)} g_K'(t) dt \le \int_{t_K}^{r(L)} g_L'(t) dt,
$
which gives
\begin{equation}
    \label{Eq:T0}
    g_K(r(L)) \le g_K(t_K) - g_L(t_K),
\end{equation}
where we used that $g_L(r(L)) = 0$.
Let us estimate the right-hand side in the last inequality. We have $g_K(t_K) - g_L(t_K) = (g_K(t_K) - g_L(t_L)) + (g_L(t_L) - g_L(t_K))$. For the second term, using $t_L \ge (1 - \eps)t_K$ and the monotonicity of $g_L$, we obtain:
\[
g_L(t_L) - g_L(t_K) \le g_L((1 - \eps)t_K) - g_L\left(t_K\right) \le \max_{\xi \in [0, r(L)]}\left(-g_L'(\xi)\right) \cdot \eps \cdot t_K \le \eps \cdot t_K \cdot |\partial K|
\]
where we used that $|\partial K| = |\partial L|$. For the first term,
\[
g_K(t_K) - g_L(t_L) = \pi \cdot (t_K^2 - t_L^2) \le 2\pi \cdot \eps \cdot t_K \cdot r(K).
\]
where we again used the fact $t_L \ge (1 - \eps)t_K$ and that $t_K+t_L \le 2 r(K)$.

Putting our estimates together in \eqref{Eq:T0},
\begin{equation}
    \label{Eq:T1}
    g_K(r(L)) \le \eps \cdot t_K \cdot \left(2\pi \cdot r(K) + |\partial K|\right).
\end{equation}

Now, using the inclusion \cite[Eq. (2.4)]{HS} 
\[
\frac{r(K) - r(L)}{r(K)} K \subset K_{r(L)},
\]
we obtain 
\[
\left(\frac{r(K)-r(L)}{r(K)}\right)^2 |K| \le g_K(r(L)).
\]
Recall that $\Rr(K) = r(K)/r(L)$. Hence, the last estimate gives
\[
\Rr(K) - 1 \le \Rr(K) \cdot \sqrt{\frac{g_K(r(L))}{|K|}},
\]
and solving for $\Rr(K)$, we get
\begin{equation}
\label{Eq:1}
    \Rr(K)-1 \le \frac{\sqrt{\frac{g_K(r(L))}{|K|}}}{1-\sqrt{\frac{g_K(r(L))}{|K|}}}.
\end{equation}

Let us estimate ${g_K(r(L))}/{|K|}$ from above. Note that this quotient is at most $1$. Using \eqref{Eq:T1}, we get
\[
\frac{g_K(r(L))}{|K|} \le \eps \cdot t_K \cdot \left(\frac{2\pi r(K)}{|K|} + \frac{|\partial K|}{|K|}\right) \le 2 \eps \cdot t_K \frac{|\partial K|}{|K|} \le \frac{2 \eps}{1 - \eps} \cdot \frac{t_L \cdot |\partial L|}{|L|}, 
\]
because $2\pi r(K) \le |\partial K|$, $|\partial K|/|K| \le |\partial L|/|L|$ and $t_L \ge (1 - \eps)t_K$. Hence, if ${t_L \cdot |\partial L|}/{|L|} \le A$ for some $A > 0$, then in \eqref{Eq:1}, for sufficiently small $\eps \le \eps_0$ (e.g., any $\eps_0 < 1/(2 A)$ will do),
\[
 \Rr(K)-1 \le  \sqrt{\eps} \, \frac{\sqrt{2 A}}{\sqrt{1 - \eps_0} - \sqrt{2 \eps_0 A}}.
\]
Thus, to finish the proof of stability in the reverse Cheeger inequality using Theorem~\ref{Thm:B}, it is enough to show that $\left(t_L |\partial L|\right) / |L|$ is bounded above. 

Let us first show that $\lim_{|\partial L| \to 0} (t_L \cdot |\partial L| / |L|)$ exists. Indeed, since the $g_L(t)$ is a convex function on $[0,r(L)]$, we have 
\[
t_L \le t',
\]
where $t'$ is the first coordinate of the unique intersection of the graph $t \mapsto \pi t^2$ with the line connecting $(0, |L|)$ and $(r(L), 0)$. This line has the equation $t \mapsto -|L|/r(L) \cdot t + |L|$. Thus,
\[
t' = \frac{1}{2\pi} \left(-\frac{|L|}{r(L)} + \sqrt{\frac{|L|^2}{r(L)^2}+4\pi|L|}\right).
\]
By direct computation,
\begin{equation}
\label{Eq:LensBasic}
|L| = \frac{|\partial L|}{2} - \sin \frac{|\partial L|}{2}, \quad r(L) = 1 - \cos\frac{|\partial L|}{4}, 
\end{equation}
and thus, as $|\partial L| \to 0$,
\begin{equation}
    \label{Eq:LAss}
    |L| = \frac{|\partial L|^3}{48}\left(1 + o(1)\right), \quad r(L) = \frac{|\partial L|^2}{32}\left(1 + o(1)\right).
\end{equation}
Hence, as $|\partial L| \to 0$,
\[
t' = \frac{|\partial L|^2}{32}\left(1 + o(1)\right).
\]
Altogether, as $|\partial L| \to 0$,
\[
\frac{t_L \cdot |\partial L|}{|L|} \le \frac{t' \cdot |\partial L|}{|L|} \le \frac{3}{2} \cdot \frac{|\partial L|^2 \cdot |\partial L|}{|\partial L|^3}\left(1 + o(1)\right) = \frac{3}{2} + o(1). 
\]
Hence, there exists $\delta >0$ (explicitly computable) such that
\[
\frac{t_L \cdot |\partial L|}{|L|} \le 2, \quad \text{ if }\quad |\partial L| < \delta.
\]
On the other hand, if $|\partial L| \ge \delta$, then by \eqref{Eq:LensBasic}, $|L| \ge c(\delta) := \delta/2 - \sin(\delta/2)$, and thus 
\[
\frac{t_L \cdot |\partial L|}{|L|} \le \frac{1 \cdot 2\pi}{c(\delta)} = \frac{2\pi}{c(\delta)}.
\]
Therefore, for every $1$-convex lens $L \subset \R^2$,
$
{t_L \cdot |\partial L|}/{|L|} \le \max\{2, {2\pi}/{c(\delta)}\},
$
which is the desired universal bound. We remark that a sharper bound can be obtained using inequality (5) in \cite{Ftohi}.

%%%%%%%%%%%%%%%%%%%%%%%%%%%%%%%%%%%%%%%%%%%%%%%%%
\section{Proof of Theorem~\ref{Thm:B}}
\label{Sec:ProofB}
%%%%%%%%%%%%%%%%%%%%%%%%%%%%%%%%%%%%%%%%%%%%%%%%%

In this section, we will prove Theorem~\ref{Thm:B}. We begin by listing some preliminaries and recalling the needed construction from \cite{DrTa}.

Without loss of generality, we assume that $\lambda=1$. Furthermore, by Blaschke's rolling theorem (see, e.g., \cite{DrTa}), it is enough to prove the statement when $K$ is a \emph{$1$-convex polygon}, i.e., the intersection of finitely many disks of radius $1$. The rest would follow by approximation \cite[Proposition 2.5]{DrTa}. 

\subsection{The basic construction and preliminaries}\label{Subsect:basic}

We will need the following property of the Hausdorff distance.

\begin{lemma}[Hausdorff distance from pieces]
\label{Lem:HPieces}
Let $\Omega, \Omega' \subset \R^n$ be a pair of convex bodies that are decomposed into $N$ convex sets
\[
\Omega = \bigcup_{i = 1}^N \Omega_i, \quad \Omega' = \bigcup_{i = 1}^N \Omega'_i
\]
so that $\dH(\Omega_i, \Omega_i') \le \delta$ for every $i$. Then $\dH(\Omega, \Omega') \le \delta$. \qed
\end{lemma}

Furthermore, we will need the following straightforward lemma on the proximity of lenses.

\begin{lemma}[Proximity of lenses]
\label{Lem:Proximity}
There exist $\eps_0 \ge 0$ and $T \ge 1$ so that the following holds for each positive $\eps \le \eps_0$: 

If $L, L' \subset \R^2$ are two $1$-convex lenses such that 
\[
1 \le \frac{r(L)}{r(L')} \le 1 + \eps,
\]
then 
\[
1 \le \frac{|\partial L|}{|\partial L'|} \le 1 + T \eps.
\]
\end{lemma}

\begin{proof}
For simplicity of notation, let us put $x := r(L)$, $x' := r(L')$. By \eqref{Eq:LensBasic}, $|\partial L|  = 4 \arccos(1 - r(L)) = 4 \arccos(1 - x) =: f(x)$. If $x' \ge \delta$ for some $\delta > 0$, then $x \ge x' \ge \delta$ and $f(x') \ge f(\delta) = 4 \arccos(1 - \delta)$. Also, $f(x) - f(x') \le Q \cdot (x - x') \le Q \eps$, where $Q$ is the Lipschitz constant of $f(x)$ for $x \in [\delta, 1]$. Therefore, we will get
\[
\frac{f(x)}{f(x')} - 1 \le \frac{Q \eps}{f(\delta)} = T\eps, \quad \text{where}\ \  T= \frac{Q}{f(\delta)}.
\]
In order to prove the result, we need to study the asymptotics of $f(x)/f(x')$ as $x' \to 0$ (and hence, as $x \to 0$). For sufficiently small $x'$, we have:
\[
\frac{f(x)}{f(x')} - 1 = \frac{f(x) - f(x')}{f(x')} \le \frac{f'(x')(x-x')}{f(x')} \le \eps \cdot \frac{f'(x')\cdot x'}{f(x')}.
\]
It is straightforward to check that $\lim_{x' \to 0}{f'(x')\cdot x'}/{f(x')} = 1/2$. Thus, there exists $\delta > 0$ so that, say, ${f'(x')\cdot x'}/{f(x')} \le 1$ for all $0 < x' < \delta$. With such $\delta$, altogether we obtain
\[
\frac{f(x)}{f(x')} - 1 \le \eps \cdot \max \left\{1, \frac{Q}{f(\delta)}\right\},
\]
so the claim follows.
\end{proof}

Let $L$ be a $1$-convex lens such that $r(L) = r(K)$. We denote by $B$ the inscribed ball in $K$ and assume that $B$ is also the inscribed ball for $L$. Rotate $L$ so that $\partial L$ and $\partial K$ share at least one common tangent point to $\partial B$. Let $r = r(K) = r(L)$ be the radius of $B$ and $o$ be its center.

If $L'$ is any $1$-convex lens such that $|\partial K| = |\partial L'|$, then $\mathcal{G}(|\partial K|) = r(L')$ and $\Rr(K) = r(K)/r(L')$. By the assumption of Theorem~\ref{Thm:B} and the reverse inradius inequality, we know that $1 \le {r(K)}/{r(L')} \le 1 + \eps$. Thus, $1 \le {r(L)}/{r(L')} \le 1 + \eps$. By Lemma~\ref{Lem:Proximity}, there exists a constant $T > 1$ such that 
\[
1 \le \frac{|\partial L|}{|\partial L'|} = \frac{|\partial L|}{|\partial K|} \le 1 + T \eps.
\]
Thus, denoting $\eps' := T \eps$, our standing assumption is
\begin{equation}
\label{Eq:Equivalent}
r(K) = r(L) = r, \quad 1 - \eps' \le \frac1{1+\eps'} \le \frac{|\partial K|}{|\partial L|} \le 1.
\end{equation}

By Blaschke's rolling theorem, $r \le 1$. If $r \ge 1 - \sqrt{\eps'}$, then, again by Blaschke's rolling theorem, 
\[
\dH(K, B_1) \le 2 \sqrt{\eps'}, \quad \dH(L,B_1) \le 2 \sqrt{\eps'},
\]
where $B_1$ is any common unit tangent ball to $K$ and $L$ that contains both $K$ and $L$. Hence, by the triangle inequality,
\[
\dH(K,L) \le 4 \sqrt{\eps'} = 4 \sqrt{T} \cdot \sqrt{\eps},
\]
and the claim of the theorem follows with the constant $C = 4\sqrt{T}$. For the remainder of the section, we will assume that
\begin{equation}
\label{Eq:Eps}
r \le 1 - \sqrt{\eps'}.
\end{equation}

Let us recall several constructions from \cite[Section 4]{DrTa} (adapted to $\R^2$). We denote by  
\[
\Pi \colon \R^2 \sm \{o\} \to \partial B 
\]
the radial projection from the center of $B$ to its boundary. For each side $F_i$ of $K$, let $\tilde F_i := \Pi(F_i) \subset \partial B$ be its projection on the inscribed circle. For each side $F$ of $L$, we similarly define $\tilde F = \Pi(F)$. By \cite[Claim 4.3]{DrTa}, for every side $F_i$ of $K$ that touches $B$,
\begin{equation}
\label{Eq:KeyClaim}
\frac{|F_i|}{|\tilde F_i|} \le \frac{|F|}{|\tilde F|}.
\end{equation}
Our main goal will be to study a quantitative version of this inequality.

Note that the lenses in $\R^2$ and in $\R^n$ for $n \ge 3$ have different combinatorial structure. Namely, in dimension $2$, an `edge of the lens' consists of two points, and hence is disconnected, while in dimensions $n \ge 3$, an edge is the intersection of two spherical caps, and hence is connected.

%%%%%%%%%%%%%%%%%%%%%%%%%%%%%%%%%%%%%%%%%%%%%%%%%
\subsection{Conditional proof of Theorem~\ref{Thm:B}} 
\label{SSec:UnderAss}
%%%%%%%%%%%%%%%%%%%%%%%%%%%%%%%%%%%%%%%%%%%%%%%%%

Assume first that 
\begin{equation}
\label{Eq:Touching}
\text{each side of $K$ touches $B$}.
\end{equation}
Later we will show how to reduce the general problem to this situation. In $\R^2$, each side $F_i$ is an arc of a unit circle that touches the inscribed disk $B$. We will call them \emph{sides}, and suppose there are $N$ of them. For technical reasons, we will assume that each side starts at a vertex of $\partial K$ and terminates at the touching point. In this way, $N \ge 4$ and each lens has four isometric sides, each denoted by $F$. To each side $F_i$, we can associate the angle $\phi_i$ in $B$ that spans $\tilde F_i$. Note that $F$ corresponds to the angle $\pi/2$. Define a function
\[
H(\phi_i) := \frac{|F_i|}{|\tilde F_i|} = \frac{|F_i|}{r \phi_i}.
\] 
We can naturally extend $H$ to all angles $\phi \in [0, \frac{\pi}2]$ by defining it as the ratio of the length of an arc of a unit circle subtended by an angle $\phi$ and touching $B$,  to the length of its radial projection, namely $r\phi$. In $\mathbb{R}^2$, $H(\phi)$ can be written explicitly as
\[
H(\phi) = \frac{1}{r} - \frac{\arcsin\big((1-r)\sin \phi \big)}{r \cdot \phi},
\]
and $H(0) = 1.$ By \eqref{Eq:KeyClaim}, %the function $H(\phi_i)$ satisfies
\begin{equation}
\label{Eq:Bound}
H(\phi_i) \le H(\pi/2) \qquad \text{for all} \  i \in \{1, \ldots, N\}.
\end{equation}
By assumption~\eqref{Eq:Touching}, we have
\begin{equation}
\label{Eq:Sum}
\sum_{i = 1}^N \phi_i = 2\pi.
\end{equation}
Furthermore, by \eqref{Eq:Equivalent}, 
\[
4|F| (1 - \eps') \le \sum_{i=1}^N |F_i|.
\]
Hence,
\[
4|F| (1 - \eps') \le \sum_{i=1}^N r \phi_i \cdot H(\phi_i) \overset{\eqref{Eq:Bound}}{\le} 2 \pi r \cdot H(\pi/2) = 4 |F|.
\]
Therefore, $\sum_{i=1}^N r \phi_i \cdot \big(H(\pi/2) - H(\phi_i)\big) \le \eps' \cdot 4|F|$, and hence,
\begin{equation}
\label{Eq:Alternative}
\sum_{i=1}^N \phi_i \cdot (H(\pi/2) - H(\phi_i)) \le \frac{\eps' \cdot 4|F|}{r}. 
\end{equation}

Our goal now is to conclude that only four angles
can be large, and thus of size comparable to $\pi/2$, while the other angles must be of order $\eps |F|$.

By \eqref{Eq:LensBasic}, we have
\[
|F| = \arccos(1-r).
\]

Observe that the function $H$ is monotonically increasing on $[0, \pi/2]$. 
Then there exists a suitable constant $C = C(r)>0$ such that 
\begin{equation}
\label{Eq:Side1}
\text{if } \phi \le \pi/6, \text{ then }H(\pi/2)-H(\phi) \ge 1/C(r). 
\end{equation}
One can compute explicitly that we can choose
\[
C(r) = \frac{\pi}{2} \cdot \frac{r}{3\arcsin\left(\frac{1-r}{2}\right) -\arcsin(1-r)}.
\]
Furthermore, since the function $H(\phi)$ is convex on $[0, \pi/2]$, we can deduce that
\[
H(\phi) \le \frac{2\big(\pi(1 - r) - 2\arcsin(1-r)\big)}{\pi^2 r} \phi + 1,
\]  
where the right-hand side is the equation of the line through points $(0, H(0))$ and $(\frac{\pi}2, H(\frac{\pi}2))$. Therefore, there exists a suitable constant $D = D(r) > 0$ such that for every $0< \delta < H(\pi/2)-H(0)$,
\begin{equation}
\label{Eq:Side2}
\forall \phi \in [0, \pi/2] \text{ such that }H(\pi/2) - H(\phi) \le \delta, \text{ it follows that }\pi/2- \phi \le \delta \cdot D(r).  
\end{equation}
Explicitly, we can choose
\[
D(r) = \frac{\pi^2 r}{2\big(\pi(1 - r) - 2\arcsin(1-r)\big)}.
\]

Now we are ready to analyze \eqref{Eq:Alternative}. Since each term in the sum is non-negative, we have the following estimate:
\begin{align*}
    \frac{\eps' \cdot 4|F|}{r} \ge \sum_{i=1}^N \phi_i \cdot (H(\pi/2) - H(\phi_i)) \ge  \sum_{\phi_i \le \frac{\pi}6} \phi_i \cdot (H(\pi/2) - H(\phi_i)) \ge \sum_{\phi_i \le \frac{\pi}6} \phi_i \, \frac1{C(r)}
\end{align*}
where we used \eqref{Eq:Side1}. Hence,
$$
\sum_{\phi_i \le \frac{\pi}6} \phi_i \le \eps' \cdot 4|F| \cdot \frac{C(r)}{r} = \eps' \cdot \tilde C(r)
$$
where $\tilde C(r) : =4|F| \cdot \frac{C(r)}{r}.$ Now we bound $\tilde C(r)$. Explicitly, we have
\[
\tilde C(r) = \frac{2\pi \arccos(1-r)}{3\arcsin\left(\frac{1-r}{2}\right) -\arcsin(1-r)},
\]
which is a strictly increasing function on $[0, 1-\sqrt{\eps'}]$ (here we are taking into account the bound~\eqref{Eq:Eps}). Hence,
\[
\tilde C(r) \le \tilde C\big(1-\sqrt{\eps'}\big) \le \frac{2\pi^2}{\sqrt{\eps'}}
\] 
(the last estimate follows by a straightforward tracing of the asymptotics at $0$ and noting that the function $\tilde C(1-\sqrt{\eps'})$ is monotone in $\eps'$). Combining all together, we obtain the following: 
\begin{equation}
\label{Eq:BS1}
\sum_{\phi_i \le \frac{\pi}6} \phi_i \le 2 \pi^2 \sqrt{\eps'}, \quad \textup{and} \quad \text{if } \phi_i \le \frac{\pi}{6}, \text{ then }\phi_i \le 2 \pi^2 \sqrt{\eps'}.
\end{equation}

On the other hand, if $\phi_i \ge \pi/6$, then 
\[
\sum_{\phi_i \ge \frac{\pi}6}  (H(\pi/2) - H(\phi_i)) \le \eps' \cdot \frac{24 |F|}{\pi r},
\]
and thus, $H(\pi/2) - H(\phi_i) \le \eps' \cdot \frac{24 |F|}{\pi r}$. By \eqref{Eq:Side2},
\[
\pi/2 - \phi_i \le \eps' \cdot D(r) \cdot \frac{24 |F|}{\pi r} = \eps' \cdot \tilde D(r), 
\]
where $ \tilde D(r) := D(r) \, \frac{24 |F|}{\pi r}$. Similarly to $\tilde C(r)$, now we bound $\tilde D(r)$. We have
\[
\tilde D(r) = \frac{12\pi \arccos(1-r)}{\pi(1 - r) - 2\arcsin(1-r)},
\]
which is a strictly increasing function, and hence for $r \in [0, 1-\sqrt{\eps'}]$,
\[
\tilde D(r) \le \tilde D\big(1-\sqrt{\eps'}\big) \le \frac{6\pi^2}{\pi - 2} \cdot \frac{1}{\sqrt{\eps'}}
\]
(again, the last estimate follows by the monotonicity and a straightforward analysis of asymptotics at $0$ for $\tilde D(1-x)$). Altogether, this yields the following:
\begin{equation}
\label{Eq:BS2}
\text{if }\phi_i \ge \frac{\pi}{6}, \text{ then }\phi_i \ge \frac{\pi}{2} - \frac{6\pi^2}{\pi - 2} \cdot \sqrt{\eps'}.
\end{equation}

Choosing the largest constants and returning back to the original $\eps$, we have obtained the following: there exists a constant $\tilde T > 1$ such that for each $i \in \{1, \ldots, N\}$,
\[
\text{either}\quad \phi_i \le \sqrt{\eps} \cdot \tilde T, \quad \text{or} \quad\phi_i \ge \frac{\pi}{2} - \sqrt{\eps} \cdot \tilde T.
\]
If the first estimate is true, we call the angle $\phi$ \emph{small}; if the second estimate is true, we call the angle $\phi_i$ \emph{big}.

Note that angles $\phi_i$ must satisfy \eqref{Eq:Sum} and also $\phi_i \in (0, \pi/2]$. If all but at most three of them are small, then the sum of the angles is at most
\[
\sum_{i=1}^N \phi_i \le \sum_{\phi_i \le \frac{\pi}6} \phi_i + 3 \cdot \frac{\pi}{2} \overset{\eqref{Eq:BS1}}{\le}  \sqrt{\eps} \cdot \tilde T + 3 \cdot \frac{\pi}{2} < 2\pi
\]
for $\eps$ sufficiently small (e.g., if we choose $\eps \le \eps_0 < \pi^2/ (4 \tilde T^2)$) which is impossible due to \eqref{Eq:Sum}. Therefore, at least four angles must be big.  

On the other hand, if at least five angles are big, then
\[
\sum_{i=1}^N \phi_i \ge \frac{5\pi}{2} - 5 \cdot \sqrt{\eps} \cdot \tilde T = 2\pi + \frac{\pi}{2} - 5 \cdot \sqrt{\eps} \cdot \tilde T> 2\pi,
\]
again, for $\eps$ sufficiently small (e.g., if $\eps \le \eps_0 < \pi^2 / (100 \,\tilde T^2)$) which is again impossible due to \eqref{Eq:Sum}. Thus, exactly four angles are large, while the remaining $N-4$ angles are small.

Without loss of generality, we can assume that $\phi_1$ is big. In this case, either $\phi_2$ or $\phi_N$ must be big too. Without loss of generality, we can assume that $\phi_2$ is big. Further assume that angles $\phi_s$ and $\phi_{s+1}$ are big for some $s \in \{3, \ldots, N-1\}$.

For each angle $\phi_i$, denote the intersection of $K$ with the closed cone over $F_i$ with the vertex at $o$ by $\Omega_i$ (recall that $o$ is the center of the inscribed disk). In this way, $\Omega_i$ is the curvilinear triangle, one of its sides being $F_i$ and $\phi_i$ being the opposite angle to this side. We have $K = \displaystyle\bigcup_{i=1}^N \Omega_i$.

We perform a similar construction for the lens $L$, namely, we decompose $L$ into four curvilinear triangles $\tilde \Omega_j$, $j \in \{1,\ldots, 4\}$ with the right angles,  numbered in such a way that $\tilde \Omega_1 \cup \tilde \Omega_2$ contains one vertex of the lens, while $\tilde \Omega_3 \cup \tilde \Omega_4$ contains the other vertex. 

By construction, both $\Omega_1 \cup \Omega_2$ and $\Omega_s \cup \Omega_{s+1}$ are $\sqrt{\eps}$-close to the `half' of the lens, which implies the following claim.

\begin{claim}
There exists a constant $C > 0$ and a rotation of the lens $L$ around $o$ such that after the rotation
\[
\dH\big(\tilde \Omega_1 \cup \tilde \Omega_2, \Omega_1 \cup \Omega_2\big) \le C \sqrt{\eps}, \quad \dH\big(\tilde \Omega_3 \cup \tilde \Omega_4, \Omega_s \cup \Omega_{s+1}\big) \le C \sqrt{\eps}.
\] 
\end{claim}

Now, the theorem follows by Lemma~\ref{Lem:HPieces}, because $\displaystyle\bigcup_{i =3}^{s-1} \Omega_i$ and $\displaystyle\bigcup_{i =s+1}^{N} \Omega_i$ are $C'\sqrt{\eps}$-close to a pair of segments (for some other constant $C'>0$).

%%%%%%%%%%%%%%%%%%%%%%%%%%%%%%%%%%%%%%%%%%%%%%%%%%%%%
\subsection{Proof of Theorem~\ref{Thm:B} without assuming~\eqref{Eq:Touching}}
%%%%%%%%%%%%%%%%%%%%%%%%%%%%%%%%%%%%%%%%%%%%%%%%%%%%%

Let $K$ be a $1$-convex polygon, and $B$ be its inscribed disk of radius $r$. Starting with $K$, we build a $1$-convex polygon for which \eqref{Eq:Touching} is satisfied as follows. The polygon $K$ is the intersection of finitely many unit disks. Remove from this intersection those disks that do not touch $B$. In this way, we will get a polygon $\tilde K$ such that
\begin{equation}
\label{Eq:InitialConstruction}
r(K) = r(\tilde K) = r, \quad |\partial K| < |\partial \tilde K|,
\end{equation}
and that satisfies Assumption \eqref{Eq:Touching}.

Since $|\partial K| < |\partial \tilde K|$, we have that $\mathcal G(|\partial K|) < \mathcal G(|\partial \tilde K|)$. Therefore,
\[
\Rr\big(\tilde K\big) = \frac{r}{\mathcal G(|\partial \tilde K|)}  < \frac{r}{\mathcal G(|\partial K|)} = \Rr (K) \le 1 + \eps.
\]

By the result above, we can choose a $1$-convex lens $L$ so that
\begin{equation}
\label{Eq:Est}
r(L) = r, \quad 1 - \eps \le \frac{|\partial K|}{|\partial L|} \le \frac{|\partial \tilde K|}{|\partial L|} \le 1,
\end{equation}
and 
\begin{equation}
\label{Eq:Cor}
\dH\big(\tilde K, L\big) \le C \sqrt{\eps}.
\end{equation}

Thus, by \eqref{Eq:Est} and \eqref{Eq:InitialConstruction}, $1 - \eps \le {|\partial K|}/{|\partial \tilde K|} < 1$, and hence 
\begin{equation}
\label{Eq:Starting}
|\partial \tilde K| - |\partial K| \le \eps \cdot |\partial \tilde K|.
\end{equation}
Using these estimates, we show that 
\begin{equation}
\label{Eq:Goal}
\dH(K, \tilde K) \le C \cdot \sqrt{\eps},
\end{equation}
for some constant $C > 0$. Together with \eqref{Eq:Cor}, this will yield the conclusion of Theorem~\ref{Thm:B}.

By construction, $K \subset \tilde K$. Consider the closure $T$ of a connected component of $\tilde K \sm K$. It follows that $T$ is a curvilinear non-convex polygon bounded by circular arcs of radius $1$. Exactly two adjacent sides of $T$ are coming from $\partial \tilde K$; call these sides $\ell_-$ and $\ell_+$. By \eqref{Eq:Starting},
\begin{equation}
\label{Eq:BE}
|\ell_- \cup \ell_+| - |\partial T \sm (\ell_- \cup \ell_+)| \le \eps \cdot |\partial \tilde K| < 2\pi \eps. 
\end{equation}
In order to prove~\eqref{Eq:Goal}, it is enough to conclude that for each such component $T$,
\begin{equation}
    \label{Eq:Goal2}
    \dH(\ell_- \cup \ell_+, \partial T \sm (\ell_- \cup \ell_+)) \le C' \cdot \sqrt{\eps},
\end{equation}
for some  constant $C' > 0$.

We need the result proven in \cite[Section 6.7.2]{BZConvex} which we state in its weaker form that is enough for our purposes:

\begin{lemma}
\label{Lem:BZ}
    Let $\Omega \subset \R^2$ be a convex body, and $\gamma \subset \partial \Omega$ be a curve connecting two boundary points $x, y \in \partial \Omega$. Assume that the integral curvature of $\gamma$ is at most $\pi$. Let $\gamma' \subset \R^2 \sm \inter \Omega$ be the rectifiable arc connecting $x$ and $y$, and otherwise disjoint from $\partial \Omega$, so that $\gamma' \cup [x,y]$ bounds a convex domain. Then there exists a constant $C > 0$ such that
    \[
    |\gamma'|^2 \ge |\gamma|^2 \cdot \left(1 + C \cdot \left(\dH(\gamma', \gamma)\right)^2\right) + 4 \left(\dH(\gamma', \gamma)\right)^2. 
    \]%\qed
\end{lemma}

From this lemma, we see that
\[
\dH(\gamma', \gamma) \le \frac{1}{2} \cdot \sqrt{|\gamma'|^2 - |\gamma|^2}.
\]
We apply Lemma~\ref{Lem:BZ} by setting $\gamma' = \ell_- \cup \ell_+$ and $\gamma = \partial T \sm (\ell_- \cup \ell_+)$. It is easy to verify that these curves satisfy the assumptions of the lemma. Thus, using \eqref{Eq:BE} and the fact that we work with curves of length at most $2\pi$, we obtain
\begin{equation*}
    \begin{aligned}
        &\dH(\ell_- \cup \ell_+, \partial T \sm (\ell_- \cup \ell_+)) \le \\ 
        &\le \frac{1}{2} \cdot \sqrt{|\ell_- \cup \ell_+| - |\partial T \sm (\ell_- \cup \ell_+)|} \cdot \sqrt{|\ell_- \cup \ell_+| + |\partial T \sm (\ell_- \cup \ell_+)|} \le\\
        &\le \frac{1}{2} \cdot \sqrt{2\pi \eps} \cdot \sqrt{4\pi} \le \pi\sqrt{2} \cdot \sqrt{\eps}.
    \end{aligned}
\end{equation*}
This gives the desired estimate \eqref{Eq:Goal2} for each connected component $T$ with $C' = \pi \sqrt{2}$ and finishes the proof that we can remove assumption \eqref{Eq:Touching}. Together with the result of Section~\ref{SSec:UnderAss}, this finishes the proof of Theorem~\ref{Thm:B}.

\end{document}